\documentclass[a4paper,11pt,intlimits,oneside]{amsart}
\usepackage{amsmath}
\usepackage{amsthm}
\usepackage{latexsym}
\usepackage{amssymb}
\usepackage{mathrsfs}
\usepackage{xcolor}
\usepackage{graphicx}
\usepackage[colorlinks=true]{hyperref}
\usepackage{tikz}
\numberwithin{figure}{section}
\def\R{{\mathbb R}}
\def\C{{\mathbb C}}

\def\T{{\mathbb T}}
\def\Z{{\mathbb Z}}

\def\la{\langle}
\def\ra{\rangle}
\def\s{\vskip 0.25cm\noindent}
\def\e{\varepsilon}

\def\phe{\varphi}
\def\build#1_#2^#3{\mathrel{
\mathop{\kern 0pt#1}\limits_{#2}^{#3}}}
\def\td_#1,#2{\mathrel{\mathop{\build\longrightarrow_{#1\rightarrow #2}^{}}}}
\DeclareFontFamily{U}{MnSymbolC}{}
\DeclareSymbolFont{MnSyC}{U}{MnSymbolC}{m}{n}
\DeclareFontShape{U}{MnSymbolC}{m}{n}{
    <-6>  MnSymbolC5
   <6-7>  MnSymbolC6
   <7-8>  MnSymbolC7
   <8-9>  MnSymbolC8
   <9-10> MnSymbolC9
  <10-12> MnSymbolC10
  <12->   MnSymbolC12}{}
\DeclareMathSymbol{\intprod}{\mathbin}{MnSyC}{'270}
\newtheorem{theorem}{Theorem}
\newtheorem{corollary}{Corollary}

\newtheorem{lemma}{Lemma}
\newtheorem{remark}{Remark}

\begin{document}

\title[The zero dispersion limit for Benjamin--Ono]{The zero dispersion limit for the Benjamin--Ono equation on the  line}
\author[P. G\'erard]{Patrick G\'erard}
\address{Universit\'e Paris--Saclay, Laboratoire de Math\'ematiques d'Orsay,  CNRS,   UMR 8628,  91405 Orsay, France} \email{{\tt patrick.gerard@universite-paris-saclay.fr}}


\date{July 24,  2023}

\begin{abstract}
We identify the zero dispersion limit of a solution of the  Benjamin--Ono equation on the line corresponding to every initial datum in $L^2(\R)\cap L^\infty(\R )$. We  infer a maximum principle and a local smoothing property for this limit. The proof is based on an explicit formula for the Benjamin--Ono equation and on the combination of  calculations in the special case of rational initial data, with approximation arguments. We also investigate the special case of an initial datum equal to the characteristic function of a finite interval, and prove the lack of semigroup property for this zero dispersion limit.
\end{abstract}


\thanks{ It is a pleasure to thank R.~Badreddine, D.~Bambusi,  Y.~Brenier, M.~Gallone, L.~Gassot, H.~Holden,  E.~Lenzmann, P.~Miller and N.~Tzvetkov for fruitful discussions about this problem.}

\maketitle

%\tableofcontents

\medskip

\section{Introduction}

We consider the Benjamin--Ono equation on the line with a small dispersion $\e >0$,
\begin{equation}\label{BOeps}
\partial_tu^\e +\partial_x[(u^\e )^2]=\e \partial_x|D| u^\e \ ,\ u^\e (0,x)=u_0(x)\ ,
\end{equation}
where $u_0\in L^2(\R )$ and is real valued. It is well known \cite{IK07, MP12}
that equation \eqref{BOeps} has a global solution $u^\e \in C(\R, L^2(\R ))$, which is characterized as the strong limit of smooth solutions
with any initial data approximating the initial datum $u_0$ strongly in $L^2(\R )$, and that the $L^2$ norm $\Vert u^\e (t)\Vert_{L^2(\R )}$ is independent of the time variable $t$. \\
Our purpose is the description of the weak limit of $u^\e (t)$ in $L^2(\R )$ as $\e \to 0$.
Such a problem is delicate even for smooth data $u_0$, because  the formal limit 
of equation \eqref{BOeps} as $\e $ tends to $0$ is the inviscid Burgers--Hopf equation \cite{Bu48, Ho50}
$$\partial_tu+\partial_x(u^2)=0\ ,$$
which is well known to display singularities in finite time. In fact, numerical simulations --- see e.g. \cite{F13} --- show that, near such singularities, the solution $u^\e $ 
display very strong oscillations as $\e $ tends to $0$, so that the weak limit of $u^\e $, if it exists, is very difficult to identify.\\
This zero--dispersion limit problem was first addressed by Lax and Levermore \cite{LL83} in the case of the Korteweg--de Vries equation,
\begin{equation}\label{KdV}
\partial_tv^\e +\partial_x[(v^\e)^2]=\e \partial_x^3v^\e\ ,\ v^\e (0,x)=v_0(x)\ ,
\end{equation}
using the integrability of this equation and the inverse scattering transform for the Lax operator. Lax--Levermore's pioneering work initiated  a series of remarkable results describing  $v^\e $  for special initial data $v_0$, under various assumptions on the spectral theory of the corresponding  Lax operator, see in particular  \cite{V85, DVZ97, GK07, CG09} and references therein.
All these references have in common the use of the inverse scattering theory and the study of a related Riemann--Hilbert problem, for describing the solution $v^\e $, and consequently depend on strong assumptions on the initial data through the spectral theory of its Lax operator --- here, a Schr\"odinger operator.\\
Let us also mention recent approaches to similar problems for the Fermi--Pasta--Ulam chain \cite{GMPR22} on a wide range of time and energy scales.  We refer to \cite{GP22} for a connection to the zero--dispersion limit for the KdV equation \eqref{KdV}. \\
In the case of the Benjamin--Ono equation \eqref{BOeps} on the line, much less is known, with the notable exceptions of \cite{MX11, MW16}. Again,  these references use the inverse scattering theory for the Lax operator inherited from the integrability of the Benjamin--Ono equation.
The case of the Benjamin--Ono equation with periodic boundary conditions was studied in  \cite{Ga21, Ga23} where the inverse spectral theory of \cite{GK21} was used to establish an explicit description of the zero dispersion limit in the case of bell--shaped data. Furthermore, in \cite{Ga23}, an explicit formula proved in \cite{G23} was used to prove the existence of a zero--dispersion limit for every initial datum in $L^\infty (\T )$. This suggests that the results of \cite{G23} could also be used for the zero--dispersion limit on the line. This is the purpose of this paper, where we prove general results on the zero--dispersion limit of \eqref{BOeps} without using the inverse spectral theory of the Lax operator.
\subsection{Statement of the results}
Our first main result is the following.
\begin{theorem}\label{main}
As $\e \to 0$,  the solution $u^\e (t)$ of \eqref{BOeps} converges weakly in $L^2(\R )$ to some function $ZD[u_0](t)$ for every  $t\in \R$.
The mapping $$u_0\in L^\infty(\R )\cap L^2(\R )\mapsto ZD[u_0]$$ has the following properties.
\begin{enumerate}
\item For every $t\in \R$, $\Vert ZD[u_0](t)\Vert _{L^2}\leq \Vert u_0\Vert _{L^2}$,  and $ZD[u_0]$ is continuous on $\R $ with values in $L^2(\R )$ endowed with the weak topology.
\item If $u_0^\delta $ converges strongly to $u_0$ in $L^2(\R )$ as $\delta \to 0$, with a uniform bound on $\Vert u_0^\delta \Vert_{L^\infty}$, then, for every $t\in \R$,
$ZD[u_0^\delta ](t)$ converges to $ZD[u_0](t)$ weakly in $L^2(\R )$.
\end{enumerate}
\end{theorem}
The above theorem is already quite surprising compared to what is known for the KdV equation. Indeed, as far as we know, the existence of a weak limit for every time and for every initial datum $v_0\in L^\infty (\R )\cap L^2(\R )$ is an open problem for equation \eqref{KdV}. 
Coming back to the Benjamin--Ono equation, recall that a similar result was proved on the torus in \cite{Ga23}. In fact, on the line, one can go much further in the description of the zero dispersion limit. Under a slight additional assumption on the regularity of $u_0$ , we can relate the zero dispersion limit $ZD[u_0]$ to the multivalued solution of the inviscid Burgers--Hopf equation with initial datum $u_0$.
\begin{theorem}\label{state:somalt}
In the conditions of Theorem \ref{main}, assume moreover that $u_0$ is a $C^1$  function tending to $0$ at infinity as well as its first derivative.  For every $t\in \R$, the set   $K_t(u_0)$ of critical values of the function $$y\in \R \mapsto y+2tu_0(y)$$ is a compact subset of measure $0$.  For every connected component $\Omega $ of $K_t(u_0)^c $, there exists a nonnegative integer $\ell $ such that, for every $x\in \Omega $,  the equation $$y+2tu_0(y)=x$$ has $2\ell +1$ simple real solutions 
$$y_0(t,x)<y_1(t,x)<\cdots <y_{2\ell}(t,x)\ ,$$ and the zero dispersion limit is given by
\begin{equation}\label{somalt}
ZD[u_0](t,x)=\sum_{k=0}^{2\ell} (-1)^k u_0(y_k(t,x))\ .
\end{equation}
\end{theorem}
\begin{center}
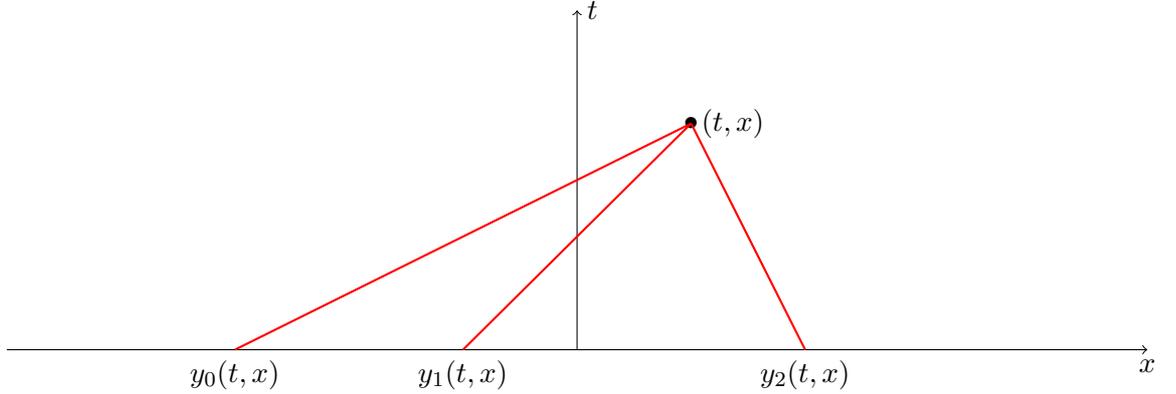
\begin{figure}
\begin{tikzpicture}[scale=1.5]
[scale=1.5]
\draw (5,0) node[below] {$x$};
\draw [->] (-5,0)--(5,0);
\draw (0,3) node[right] {$t$};
\draw [->] (0,0)--(0,3);
\draw (1,2) node[right]{$(t,x)$};
\draw (1,2) node{$\bullet$};
\draw [red][thick](-3,0)--(1,2);
\draw [red][thick](-1,0)--(1,2);
\draw [red][thick](2,0)--(1,2);
\draw (-3,0) node[below]{$y_0(t,x)$};
\draw (-1,0) node[below]{$y_1(t,x)$};
\draw (2,0) node[below]{$y_2(t,x)$};
\end{tikzpicture}
\caption{The crossing of characteristics}
\end{figure}
\end{center}
{\bf Remarks.}
\begin{enumerate}
\item Formula \eqref{somalt} was proved by Miller--Wetzel \cite{MW16} -- see also Miller--Xu \cite{MX11} --- in the special case of a rational Klaus--Shaw initial potential, and by Gassot \cite{Ga21, Ga23} in the special case of a general bell shaped initial potential with periodic boundary conditions, using inverse spectral theory for the Lax operator. Here we use a different approach and we have essentially no constraint on the initial datum $u_0$. 
\item As soon as several characteristics are crossing at a point $(t,x)$, the function  $u=ZD[u_0]$ given by \eqref{somalt} is not a weak solution of the inviscid Burgers--Hopf equation
$$\partial_tu+\partial_x (u^2)=0$$
in a neighborhood of $(t,x)$, hence the convergence of $u^\e $ to $u$ cannot be strong in $L^2_{\rm loc}$, confirming the strong oscillations detected by numerical simulations.
We refer to the end of section \ref{rational} for more detail about this lack of compactness. 
\item Formula \eqref{somalt} was introduced by Y.~Brenier in \cite{YB81, YB83, YB84} to approximate the entropic solution \cite{K69, S99} of the inviscid Burgers--Hopf equation. This formula  has the following weak version. For every test function $\phe$ on $\R$, 
\begin{equation}\label{weaksomalt} 
\int_\R ZD[u_0](t,x)\, \phe (x)\, dx =\int_\R \phe (y+2tu_0(y))u_0(y)(1+2tu_0'(y))\, dy\ .
\end{equation}
\end{enumerate}
A striking consequence of Theorem \ref{state:somalt} is the following maximum principle.
\begin{corollary}\label{maxprinciple}
For every $u_0\in L^\infty (\R )\cap L^2(\R )$,  for every $t\in \R $, 
$${\rm ess \, inf}u_0\leq {\rm ess \, inf} ZD[u_0](t)\leq {\rm ess \, sup} ZD[u_0](t)\leq {\rm ess \, sup} u_0\ .$$
\end{corollary}
Contrary to what is known about the inviscid limit of the Burgers--Hopf equation \cite{K69, S99}, where the maximum principle is inherited from the parabolic structure, the above maximum principle is surprising, because similar estimates do not hold for the solution $u^\e $ of the Benjamin--Ono equation for a given $\e >0$. In fact, for periodic boundary conditions, one can even prove that the Benjamin--Ono flow map may lose up to $1/2$ derivatives in the $C^\alpha $ regularity (see section 4 of \cite{GKT21}).
\s
Another consequence of Theorem \ref{state:somalt} is the following local smoothing property.
\begin{corollary}\label{smoothing}
For every $u_0\in L^\infty (\R )\cap L^2(\R )$,  for every $t\in [-T,T]$, the distributional derivative $2t\partial_xZD[u_0](t)$ is a measure, with a bounded total variation on every finite interval of $\R$.
In particular, for every $t\ne 0$, for every $s<\frac 12$, $ZD[u_0](t)\in H^s_{\rm loc}(\R )$.
\end{corollary}
\s
Finally, the combination of Theorem \ref{main} is and Theorem \ref{state:somalt} allows us to calculate $ZD[u_0]$ for discontinuous functions $u_0$, for instance if $u_0$ is the characteristic function of a finite interval. See section \ref{disc} below. 
\subsection{Structure of the paper.} This paper is organized as follows. In section 2, we prove Theorem \ref{main} by using an explicit formula recently derived for the initial value problem of the Benjamin--Ono on the line \cite{G23}. Assuming the initial datum $u_0$ to be $L^\infty$, it is possible to pass to the limit in this formula as $\e $ tends to $0$ and to obtain a formula for the zero--dispersion limit $ZD[u_0](t)$, using the Toeplitz operator of symbol $u_0$ acting on the Hardy space $L^2_+(\R)$,  see formula \eqref{ZDexplicit} below. In section 3, we transform formula \eqref{ZDexplicit} in the  case of a rational initial datum $u_0$ into a linear system on $\C^{N+1}$, where $2N$ is the degree of the denominator of $u_0$, and we obtain the formula \eqref{somalt} in this case. We also prove that the  weak limit to $ZD[u_0](t)$ is not strong in the case of three crossing characteristics. In section 4, we prove Theorem \ref{state:somalt} in its whole generality, and we infer Corollary \ref{maxprinciple} and Corollary \ref{smoothing}. Finally, in section 5, we describe the zero--dispersion limit $ZD[u_0]$ when $u_0$ is the characteristic function of the interval $]-1,1[$, and we prove that the map $t\mapsto ZD[\ .\ ](t)$ is not a semigroup of transformations of $L^2(\R )\cap L^\infty (\R)$. 
\section{Proof of Theorem \ref{main}}
 We first recall the notation from \cite{G23}.
Here $D$ denotes $-i\partial_x$, $\Pi $ denotes the orthogonal projector ${\bf 1}_{D\geq 0}$ from $L^2(\R )$ onto the Hardy space $$L^2_+(\R ):=\{ f\in L^2(\R ):  \forall \xi <0, \hat f(\xi )=0\} \ ,$$
where $\hat f$ denotes the Fourier transform. We recall that $L^2_+(\R )$ identifies to holomorphic fonctions $f$ in the upper half plane with 
$$\sup_{y>0}\int_\R |f(x+iy)|^2\, dx\ <+\infty \ .$$
We denote by $G$ the adjoint operator, on $L^2_+(\R )$, of the multiplication by $x$. The domain of $G$ consists of those functions $f$ in $L^2_+(\R )$ such that the restriction of $\hat f$ to $]0,+\infty [$ belongs to the Sobolev space $H^1$. Furthermore, for any function $f\in L^2_+(\R )$ such that the restriction of $\hat f$ 
to $]0,1[$ belongs to $H^1(]0,1[)$, we set 
$$I_+(f)=  \hat f(0^+)\ .$$
Finally, for $b\in L^\infty (\R )$ $T_b$ denotes the Toeplitz operator defined on $L^2_+(\R )$ by $T_b(f)= \Pi (bf)$. This operator is bounded, with $\Vert T_b\Vert =\Vert b\Vert_{L^\infty}$, and is selfadjoint if $b$ is real valued.
\vskip0.25cm
From \cite{G23}, we know that, if $w_0\in L^\infty (\R )\cap L^2(\R )$ and is real valued, the solution of the initial value problem for the Benjamin--Ono equation,
\begin{equation}\label{BO}
\partial_sw+\partial_x(w^2)=\partial_x|D|w\ ,\ w(0,x)=w_0(x)
\end{equation} 
is given by 
$$w(s,x)=\Pi w(s,x)+\overline{\Pi w(s,x)}, x\in \R \ ,$$
where, for every complex number $x$ with ${\rm Im}(x)>0$, 
\begin{equation}\label{explicitBO}
\Pi w (s,x)=\frac{1}{2i\pi} I_+[(G-2 s(D-T_{w_0}) -x{\rm Id})^{-1}(\Pi w_0)]\ .
\end{equation}
We recall \cite{G23} that the resolvent in the right hand side of  formula \eqref{explicitBO}
is well defined. Indeed, on the other hand, the operator $G-2sD$ has the following dissipation property,
$$\forall f\in {\rm Dom}(G-2sD)\ ,\  {\rm Im}\la (G-2sD)f\vert f\ra \leq 0\ ,$$
and is maximal, in the sense that $G-2sD-x{\rm Id}:{\rm Dom}(G-2sD)\to L^2_+(\R )$ is invertible for every complex number $x$ with ${\rm Im}(x)>0$.
On the other hand, the  operator $2sT_{w_0}$ is bounded and self--adjoint. The claim therefore follows from standard perturbation theory.\s
From a simple scaling argument,  the solution $u^\e $ of \eqref{BOeps} is given by 
$$u^\e (t,x)=\e w (\e t ,x)\ ,$$
where $w (s,x)$ is the solution of \eqref{BO} with $w_0 (x)=\e^{-1}u_0(x)$.
Therefore we can write
$$u^\e (t,x)=\Pi u^\e (t,x)+\overline{\Pi u^\e (t,x)}\ ,\ x\in \R \ ,$$
where, for every complex number $x$ with ${\rm Im}(x)>0$, 
$$\Pi u^\e (t,x)=\frac{1}{2i\pi} I_+[(G-2\e tD+2t T_{u_0} -x{\rm Id})^{-1}(\Pi u_0)]\ .$$
On the other hand, recall that the $L^2$ norm of $u^\e (t)$ is independent of $t$, equal to the $L^2$ norm of $u_0$, so there exists a subsequence $\e_j$ tending to $0$ such that $u^{\e _j} (t)$ has a weak limit in $L^2(\R )$. We claim that all these weak limits  coincide.
\begin{lemma}\label{state:ZDexplicit}
For every $t\in \R$, $u^\e (t)$ converges weakly in $L^2(\R )$ to $u(t)$, with
$$u(t,x)=\Pi u (t,x)+\overline{\Pi u (t,x)}\ ,\ x\in \R\ ,$$
where, for every complex number $x$ with ${\rm Im}(x)>0$, 
\begin{equation}\label{ZDexplicit}
\Pi u (t,x)=\frac{1}{2i\pi} I_+[(G+2t T_{u_0} -x{\rm Id})^{-1}(\Pi u_0)]\ .
\end{equation}
\end{lemma}
\begin{proof} Since $u\in L^2(\R )\mapsto \Pi u(x)\in \C$ is a continuous linear form for every complex number $x$ with ${\rm Im}(x)>0$, it is enough to characterize $\Pi u(t,x)$ for ${\rm Im}(x)>0$, and even for ${\rm Im }(x)$ big enough, by analytic continuation. We observe that, for ${\rm Im} (x)>2|t| \Vert u_0\Vert_{L^\infty}$,  the resolvent of $G-2\e tD+2t T_{u_0}$ converges strongly to the resolvent of $G+2tT_{u_0}$ as $\e \to 0$. Indeed, by standard perturbation theory, it is enough to check that the resolvent of $G-2\e tD$ converges to  $G$, with a uniform bound by $({\rm Im} (x))^{-1}$ if ${\rm Im} (x)>0$, and this is immediate from the following explicit formula, 
$$\hat h_\e (\xi )=i\int_\xi ^\infty \hat f(\eta ){\rm e}^{i\e t(\eta ^2-\xi ^2)+ix(\eta -\xi )}\, d\eta \ .$$
for $h_\e :=(G-2\e t D-x)^{-1}f$.\\
 Furthermore, from the equation
$$i\partial_\xi \hat f_\e (\xi )-(2\e t\xi +x)\hat f_\e (\xi )=\hat g_\e(\xi )\ ,\ g_\e :=\Pi u_0-2tT_{u_0}f_\e \ ,\ \xi >0\ ,$$
we observe that the Fourier transform of
$$f_\e =(G-2\e tD+2tT_{u_0} -x{\rm Id})^{-1}(\Pi u_0)$$
is uniformly bounded in $H^1(]0,1[)$ as $\e $ tends to $0$. \\Consequently, if ${\rm Im} (x)>2|t| \Vert u_0\Vert_{L^\infty}$, we have
$$\Pi u^\e (t,x) \longrightarrow \frac{1}{2i\pi} I_+[(G+2t T_{u_0} -x{\rm Id})^{-1}(\Pi u_0)]\ .$$
Therefore, by analytic continuation, for every $x$ such ${\rm Im}(x)>0$, 
$$
\Pi u (t,x)=\frac{1}{2i\pi} I_+[(G+2t T_{u_0} -x{\rm Id})^{-1}(\Pi u_0)]\ .
$$
which is \eqref{ZDexplicit}. From  formula \eqref{ZDexplicit}, a similar argument allows us to prove properties (1) and (2) in Theorem \ref{main}.
\end{proof} 
\begin{remark}
\begin{enumerate}
\item Notice that it is not so easy to estimate the $L^2$ norm of $u(t)$ from formula \eqref{ZDexplicit}. However, we know that it is bounded by the $L^2$ norm of $u_0$ because of the $L^2$ conservation law for the Benjamin--Ono equation.
\item The above arguments also show that, for any convergent sequence $t_\e $ of real numbers, $u^\e (t_\e )$ is weakly convergent in $L^2$ to $u(\lim _{\e \to 0}t_\e)$. Consequently, for every $T>0$, $u^\e $ is uniformly convergent to $u$ in $C([-T,T], L^2_w(\R ))$, where $L^2_w(\R )$ is $L^2(\R )$ endowed with the weak topology.
\item If the initial datum $u_0$ is replaced by a sequence $u_0^\e $ strongly convergent to $u_0$, with a uniform bound in $L^\infty $, the above arguments show that $u^\e (t)$ is still weakly convergent to $ZD[u_0](t)$.
\end{enumerate}
\end{remark}
\section{Proof of Theorem \ref{state:somalt} for a rational initial datum}\label{rational}
 In this section, we prove formula \eqref{somalt} if  $u_0$ is a rational function with real coefficients, with no pole on the real line,
$$u_0(y)=\frac{P(y)}{Q(y)}\ ,$$
where, for some positive integer $N$,  $Q$ is a monic real polynomial of degree $2N$, and $P$ is a real polynomial of degree at most $2N-1$. With no loss of generality, we may assume --- by an approximation argument in formula \eqref{ZDexplicit}---, that $Q$ has only simple poles in the complex domain, so that
$$u_0(y)=\sum_{j=1}^N \frac{c_j}{y-p_j}+\frac{\overline c_j}{y-\overline p_j}\ ,\ {\rm Im}(p_j)>0\ ,\ \Pi u_0(y)=\sum_{j=1}^N \frac{\overline c_j}{y-\overline p_j}\ .$$
Set $u(t):=ZD[u_0](t)$. We calculate directly $u(t,x)=\Pi u(t,x)+\overline{\Pi u(t,x)}$ for $x\in \R$. Recall that, according to formula \eqref{ZDexplicit}, for ${\rm Im}(x)>0$, 
$$\Pi u(t,x)=\frac{1}{2i\pi} I_+[(G+2t T_{u_0} -x{\rm Id})^{-1}(\Pi u_0)]\ .$$
Notice that the algebraic equation $y+2tu_0(y)=x$
is equivalent to
$$yQ(y)-xQ(y)+2tP(y)=0\ ,$$
which is an equation of degree $2N+1$. Hence, if $x\in \R$,  it has an odd number $2\ell (t,x)+1$ of real zeroes. Discarding only a finite set of points $x$ for a given $t$, we may consider the  case where $\ell (t,x)=\ell \in \Z_{\geq 0}$ and where these zeroes are simple, ordered as 
$$y_0(t,x)<y_1(t,x)<\dots  <y_{2\ell}(t,x)\ ,$$
so that the derivative of $y+2tu_0(y)$  at $y=y_k(t,x)$ has the sign of $(-1)^k$. Denote the other zeroes of this equation in the complex domain by $$y_m(t,x), m=2\ell +1,\dots, 2N\ ,$$ with ${\rm Im}(y_{2p}(t,x)) >0$ and $y_{2p-1}(t,x)=\overline{y_{2p}(t,x)}$ if $p=\ell +1,\dots, N$. The Cauchy--Riemann equations and the implicit function theorem show that 
$$\frac{\partial {\rm Im}(y_k(t,x))}{\partial {\rm Im}(x)}=\frac{\partial {\rm Re}(y_k(t,x))}{\partial {\rm Re}(x)}=\frac{1}{1+2tu_0'(y_k(t,x))}$$
which has the sign of $(-1)^k$.
If $x$ is shifted into the upper half plane with a small imaginary part, we infer
$${\rm Im}(y_{2k}(t,x))>0\ ,\ {\rm Im}(y_{2p-1}(t,x))<0\ ,\  k=0,\dots ,N,\  p=1,\dots, N\ .$$
For such a complex number $x$, let us calculate $$f_{t,x}:=(G+2t T_{u_0} -x{\rm Id})^{-1}(\Pi u_0)\ .$$ It turns out that $f_{t,x}$ is a rational function. Indeed, for any function $f$ in $L^2_+(\R )$, we have
$$Gf(y)=yf(y)+\frac{1}{2i\pi}I_+(f)\ ,\ T_{u_0}f(y)=u_0(y)f(y)-\sum_{j=1}^N c_j\frac{f(p_j)}{y-p_j}\ .$$
The first identity is a consequence of the Fourier representation of $G$ and of the jump formula at $\xi =0$ for distributional derivatives. The second identity comes from the following standard property of Toeplitz operators with rational symbols,
$$T_{(x-p)^{-1}}(f)(x)=\frac{f(x)-f(p)}{x-p}\ ,\ {\rm Im}(p)>0\ .$$
Hence the equation on $f_{t,x}$ can be reformulated as 
$$(y-x+2tu_0(y))f_{t,x}(y)=u_0(y)+\lambda (t,x)+\sum_{j=1}^N \frac{\mu_j(t,x)}{y-p_j}\ ,$$
where $\lambda (t,x)$, $\mu_j(t,x), j=1,\dots, N$, are complex numbers such that the rational function
$$f_{t,x}(y)=\frac{u_0(y)+\lambda (t,x)+\sum_{j=1}^N \frac{\mu_j(t,x)}{y-p_j}}{y-x+2tu_0(y)}\ ,$$
belongs to the Hardy space. This imposes the cancellation of the numerator if $y$ is any  of the zeroes of the denominator in the upper half plane, and therefore leads to the linear system
$$u_0(y_{2k}(t,x))+\lambda (t,x)+\sum_{j=1}^N \frac{\mu_j(t,x)}{y_{2k}(t,x)-p_j}=0\ ,\ k=0,\dots ,N\ , $$
or, if $t\ne 0$,
\begin{equation}\label{system}
\lambda (t,x)+\sum_{j=1}^N \frac{\mu_j(t,x)}{y_{2k}(t,x)-p_j}=\frac{y_{2k}(t,x)-x}{2t}\ ,\ k=0,\dots ,N\ ,
\end{equation}
Then we have
$$\Pi u(t,x)=\frac{1}{2i\pi}I_+(f_{t,x})=-\lim_{y\to \infty}yf_{t,x}(y)\ ,$$
since  $f_{t,x}(y)$ is a rational function of $y$. In view of the expression of $f_{t,x}$, we conclude
$$\Pi u (t,x)=-\lambda (t,x)\ .$$
From Cramer's formulae for the system \eqref{system},
\begin{eqnarray*}
\lambda (t,x)&=&\frac{N(t,x)}{D(t,x)}\ ,\\
 D&:=&\left |\begin{array}{ccccc} 1& \frac{1}{y_0-p_1}& . &. &\frac{1}{y_0-p_N}\\
1&\frac{1}{y_2-p_1}& . &. &\frac{1}{y_2-p_N}\\ 
.&.&.&.&.\\
1&\frac{1}{y_{2N}-p_1}& . &. &\frac{1}{y_{2N}-p_N}
\end{array}  \right |\ ,\\
N&:=:&\frac{1}{2t}\left |\begin{array}{ccccc} y_0-x& \frac{1}{y_0-p_1}& . &. &\frac{1}{y_0-p_N}\\
y_2-x&\frac{1}{y_2-p_1}& . &. &\frac{1}{y_2-p_N}\\ 
.&.&.&.&.\\
y_{2N}-x&\frac{1}{y_{2N}-p_1}& . &. &\frac{1}{y_{2N}-p_N}
\end{array}  \right |\ .
\end{eqnarray*}
The following lemma follows from elementary manipulations on Cauchy and Vandermonde determinants.
\begin{lemma}\label{det}
Given complex numbers $z_0,\dots ,z_N, p_1,\dots ,p_N$ pairwise distinct, we have
\begin{eqnarray*}
\frac{\left |\begin{array}{ccccc} z_0& \frac{1}{z_0-p_1}& . &. &\frac{1}{z_0-p_N}\\
z_1&\frac{1}{z_1-p_1}& . &. &\frac{1}{z_1-p_N}\\ 
.&.&.&.&.\\
z_N&\frac{1}{z_{N}-p_1}& . &. &\frac{1}{z_{N}-p_N}
\end{array}  \right |}
{\left |\begin{array}{ccccc} 1& \frac{1}{z_0-p_1}& . &. &\frac{1}{z_0-p_N}\\
1&\frac{1}{z_1-p_1}& . &. &\frac{1}{z_1-p_N}\\ 
.&.&.&.&.\\
1&\frac{1}{z_{N}-p_1}& . &. &\frac{1}{z_{N}-p_N}
\end{array}  \right |}
\ =\ \sum_{\alpha =0}^N z_\alpha -\sum_{j=1}^N p_j\ .
\end{eqnarray*}
\end{lemma}
\begin{proof}
Using the formula for the Cauchy determinants,  we have
\begin{eqnarray*}
A:=\left |\begin{array}{ccccc} z_0& \frac{1}{z_0-p_1}& . &. &\frac{1}{z_0-p_N}\\
z_1&\frac{1}{z_1-p_1}& . &. &\frac{1}{z_1-p_N}\\ 
.&.&.&.&.\\
z_N&\frac{1}{z_{N}-p_1}& . &. &\frac{1}{z_{N}-p_N}
\end{array}  \right | &=&\sum_{\alpha =0}^N (-1)^\alpha z_\alpha D_\alpha 
\\
B:=\left |\begin{array}{ccccc} 1& \frac{1}{z_0-p_1}& . &. &\frac{1}{z_0-p_N}\\
1&\frac{1}{z_1-p_1}& . &. &\frac{1}{z_1-p_N}\\ 
.&.&.&.&.\\
1&\frac{1}{z_{N}-p_1}& . &. &\frac{1}{z_{N}-p_N}
\end{array}  \right |&=&\sum_{\alpha =0}^N (-1)^\alpha D_\alpha 
\end{eqnarray*}
with 
\begin{eqnarray*}
D_\alpha &:=&\prod _{\substack{\beta \ne \alpha, \gamma\ne \alpha \\ \beta <\gamma } }(z_\beta -z_\gamma) \prod_j (z_\alpha -p_j)  \Delta \ ,\\
\Delta &:=& \frac{\prod_{j<k}(p_k-p_j)}{\prod_{\beta ,j}(z_\beta -p_j)}\ .
\end{eqnarray*}
Consequently, we are led to evaluate the following quotient of Vandermonde determinants,
\begin{eqnarray*}
\frac AB&=&\frac{V(R)}{V(Q)}\ ,\ R(z):=zQ(z)\ ,\ Q(z)=\prod_{j=1}^N (z-p_j)\ ,\\
V(P)&:=&\left |\begin{array}{ccccc} P(z_0)& 1& z_0 &. &z_0^{N-1}\\
P(z_1)&1& z_1 &. &z_1^{N-1}\\ 
.&.&.&.&.\\
P(z_N)&1& z_N&. &z_N^{N-1}
\end{array}  \right |\ .
\end{eqnarray*}
We notice that the linear form $V$ cancels on polynomials of degree  at most $N-1$, and on the polynomial $\tilde P $ defined as   $$\tilde P(z):=\prod_{\alpha =0}^N(z-z_\alpha)\ .$$
Since 
\begin{eqnarray*}
R(z)-\tilde P(z)&=&zQ(z)-\tilde P(z)=\left (\sum_{\alpha =0}^N z_\alpha -\sum_{j=1}^N p_j\right )z^N +\C_{\leq N-1}[z]\\
&=&\left (\sum_{\alpha =0}^N z_\alpha -\sum_{j=1}^N p_j\right )Q(z) +\C_{\leq N-1}[z]\ ,
\end{eqnarray*}
we infer that
$$\frac{V(R)}{V(Q)}=\frac{V(R-\tilde P)}{V(Q)}=\sum_{\alpha =0}^N z_\alpha -\sum_{j=1}^N p_j\ .$$
\end{proof}
Using lemma \ref{det}, we have
\begin{equation}\label{lambda}
\lambda (t,x)=\frac{1}{2t}\left (\sum_{\alpha =0}^{N} y_{2\alpha}(t,x) -\sum_{j=1}^Np_j-x  \right )\ .
\end{equation}
In view  of the algebraic equation $$yQ(y)-xQ(y)+2tP(y)=0\ ,\ Q(y)=\prod_{j=1}^N (y-p_j)(y-\overline p_j)\ ,$$ we have 
\begin{equation}\label{relxpy}
x+\sum_{j=1}^N (p_j+\overline p_j)=\sum_{\alpha =0}^{2N}y_\alpha \ .
\end{equation}
Now we make $x$ tend to the real line, so that   $y_{k}(t,x)$ is real for \break $k=0,1,\dots, 2\ell $, and  $y_{2p-1}(t,x)=\overline{y_{2p}(t,x)}$ if $p=\ell +1,\dots, N\ .$
Consequently, 
\begin{eqnarray*}
u(t,x)&=&-(\lambda (t,x)+\overline {\lambda (t,x)})\\
&=&\frac{1}{2t}\left ( 2x+\sum_{j=1}^N (p_j+\overline p_j)- 2\sum_{\alpha =0}^\ell y_{2\alpha }(t,x)- \sum_{\beta =2\ell +1}^{2N} y_\beta  (t,x)\right )\ ,\\
&=&\frac{1}{2t}\left (x+\sum_{\gamma =1}^{\ell}  y_{2\gamma -1}(t,x)-\sum_{\alpha =0}^{\ell} y_{2\alpha }(t,x)\right )\ ,\\
&=&\sum_{k=0}^{2\ell} (-1)^k u_0(y_k(t,x))\ .
\end{eqnarray*}
The proof of formula \eqref{somalt} when $u_0$ is rational is complete.
\s
Let us end this section by an elementary observation about formula \eqref{somalt}. We claim that, if $\ell =1$, this alternate sum of solutions of the inviscid Burgers equation is not a solution of the inviscid Burgers--Hopf equation. Indeed, let us argue by contradiction. Assume $\ell =1$
 and that $u$ satisfies the inviscid Burgers--Hopf equation  in some neighborhood $V$ of $(t,x)$ in $\R ^2$. Writing $v_k(t,x):=u_0(y_k(t,x))$ for $k=0,1,2$, we observe that $\partial _tv_k+\partial_x(v_k^2)=0$, hence 
$$\partial_x[(v_0-v_1+v_2)^2-v_0^2+v_1^2-v_2^2]=0\ $$
in $V$. We can reformulate the above identity as
$$\partial_x [(v_1-v_0)(v_1-v_2)]=0$$
on $V$. Since we are dealing with real analytic functions, this identity extends to the whole connected component of $V$ in the domain $W$ characterized by $\ell =1$. In other words, the function $(v_1-v_0)(v_1-v_2)$ is not depending on $x$ in this domain. Since this function tends to $0$ as $(t,x)$ tends to the boundary of $W$, we conclude that $(v_1-v_0)(v_1-v_2)$ is identically $0$, which is a contradiction since $y_k(t,x)+2tv_k(t,x)=x$, while $y_0(t,x)<y_1(t,x)<y_2(t,x)$.

\section{Proof of Theorem \ref{state:somalt}, of the maximum principle and of the local smoothing property}\label{C1}
We now consider the case where $u_0\in L^2(\R )$ is continuously differentiable, with $|u(x)|+|u'(x)|\to 0$ as $x\to \infty$. 
Using a standard mollifier, we may approximate $u_0$ in $L^2(\R )\cap C^1_b(\R )$ by a sequence $u_0^\delta $ of functions in $H^s(\R )$ for any $s\in \R$.
Since $L^2$ rational functions are dense in $H^2(\R )$, we may assume that $u_0^\delta $ is rational.
Let $t\in \R$, and let $K_t(u_0)$ be the set of critical values of the function $$f_t: y\in \R \mapsto y+2tu_0(y)\in \R\ .$$ By the Sard theorem, $K_t(u_0)$ has zero Lebesgue measure.  Furthermore, since $1+2tu_0'(y)\to 1$ as $y\to \infty $, the set of critical points of $f_t$  is compact, so that its image $K_t(u_0)$ by $f_t$ is compact too. Let $\Omega $ be any connected component of $K_t(u_0)^c $.
Let $\ell $ be the nonnegative integer such that, for every $x\in \Omega $, the equation $y+2tu_0(y)=x$ has $2\ell +1$ solutions, denoted by
$$y_0(t,x)<\dots <y_{2\ell}(t,x)\ .$$
Let $\omega $ be any open subinterval of $\Omega $ such that $\overline \omega $ is compact. For $\delta $ small enough, $\omega $ does not meet $K_t(u_0^\delta )^c $, and, for every $x\in \omega$,  the equation $$y+2tu_0^\delta (y)=x$$ has exacly $2\ell +1$ solutions, 
$$y_0^\delta (t,x) <\dots <y_{2\ell}^\delta (t,x)\ .$$ 
Furthermore, for every $k\in \{0,1,\dots, 2\ell\}$,  $y_k^\delta (t,x)\to y_k(t,x)$ as $\delta \to 0$. From section \ref{rational}, we know that, for $x\in \omega $,
$$ZD[u_0^\delta ](t,x)=\sum_{k=0}^{2\ell}(-1)^k u_0^\delta (y_k^\delta (t,x))\ .$$
Passing to the weak limit in $L^2(\omega )$ as $\delta \to 0$, we infer, for almost every $x\in \omega $,
$$ZD[u_0 ](t,x)=\sum_{k=0}^{2\ell}(-1)^k u_0 (y_k (t,x))\ .$$
Since $\omega $ is an arbitrary relatively compact subinterval of $\Omega $, this proves formula \eqref{somalt}. 
\s
We now come to the proof of Corollary \ref{maxprinciple}. First consider the case where $u_0$ is continuously differentiable with $|u_0(x)|+|u'_0(x)|\to 0$ as $x\to \infty$, as in section \ref{C1}.
Let us come back to formula \eqref{somalt},
$$u(t,x)=\sum_{k=0}^{2\ell} (-1)^k u_0(y_k(t,x))$$
where $y_0(t,x)<\dots <y_{2\ell}(t,x)$ are the solutions of $y+2tu_0(y)=x$. In view of the monotonicity of the sequence $k\mapsto y_k(t,x)$, we infer the monotonicity of  the sequence 
$k\mapsto u_0(y_k(t,x))$. Consequently, we have 
$$\min_{0\leq k\leq 2\ell}  u_0(y_k(t,x))\leq \sum_{k=0}^{2\ell} (-1)^k u_0(y_k(t,x))\leq \max_{0\leq 2\ell} u_0(y_k(t,x))$$
and finally 
$$\min _{y\in \R}u_0(y)\leq u(t,x)\leq \max_{y\in \R}u_0(y)\ ,$$
which is the claimed maximum principle.
\\
The general case of $u_0\in L^\infty (\R )\cap L^2(\R )$ follows by applying a standard mollifier to  $u_0$ and using property (2) in Theorem \ref{main}.
\s
Finally, let us prove Corollary \ref{smoothing}. Let us first assume that $u_0$ is continuously differentiable with $|u_0(x)|+|u'_0(x)|\to 0$ as $x\to \infty$.
 Using the weak version \eqref{weaksomalt} of formula \eqref{somalt}, we have, for every $\phe \in C^1(\R )$ with compact support,
\begin{eqnarray*}
\la 2t \partial_xZD[u_0](t), \phe \ra &=& -2t \int_\R u_0(y)\phe '(y+2tu_0(y))(1+2tu_0'(y))\, dy\\
&=&2t \int_\R u_0'(y)\phe (y+2tu_0(y))\, dy\\
&=&\int_\R \left [\frac d{dy}\left (\int_{-\infty}^{y+2tu_0(y)}\phe (s)\, ds\right ) -\phe (y+2tu_0(y))\right ]\, dy\\
&=&\int_\R \phe (s)\, ds -\int_\R \phe (y+2tu_0(y))\, dy\ .
\end{eqnarray*}
By mollifying $u_0$, this statement still holds if $u_0\in L^\infty (\R )\cap L^2(\R )$.
We infer that, for every $u_0\in L^\infty (\R )\cap L^2(\R )$,
\begin{equation}\label{derivativeZD}
2t \partial_xZD[u_0](t)=1-\mu_{2tu_0}\ ,\  \int _\R\phe (y)\, d\mu_f(y):=\int_\R \phe (y+f(y))\, dy \ .
\end{equation}
 Notice that  $\mu_f([-R,R])\leq 2(R+\| f\|_\infty )$. This proves the first claim of  Corollary \ref{smoothing}. The second claim is a consequence of the Sobolev embedding $$BV_{\rm loc}(\R )\subset H^s_{\rm loc}(\R )$$ for every $s<\frac 12$.
 \begin{remark}. Equation \eqref{derivativeZD} implies a one--sided Lipschitz condition for $ZD[u_0](t)$. This property  is well-known for entropic solutions of the inviscid Burgers--Hopf equation \cite{O57, BO88}.
 \end{remark}
\section{A special case and the lack of semigroup property}\label{disc}
In this section, we consider the case where $u_0$ is the characteristic function of a finite interval.
We are going to calculate $ZD[u_0]$ by approximating $u_0$ by a sequence of compactly supported $C^1$ 
functions. An alternative method would be to use formula \eqref{derivativeZD} above.
\begin{theorem}\label{step}
Let $u_0$ be the characteristic function of $]-1,1[$. Then $ZD[u_0]$ can be described as follows.
If $t\in ]0,1]$,
\begin{eqnarray*}
ZD[u_0](t,x)=\begin{cases} 0 & {\rm if}\ x\in ]-\infty ,-1]\cup ]2t+1,+\infty[\\
\frac{x+1}{2t} & {\rm if}\ x\in ]-1, 2t-1]\\
1 & {\rm if}\ x\in ] 2t-1, 1]\\
1-\frac{x-1}{2t} & {\rm if}\ x\in ] 1, 2t+1]\ .
 \end{cases}
 \end{eqnarray*}
 If $t\in ]1,+\infty [$,
\begin{eqnarray*}
ZD[u_0](t,x)=\begin{cases} 0 & {\rm if}\ x\in ]-\infty ,-1]\cup ]2t+1,+\infty[\\
\frac{x+1}{2t} & {\rm if}\ x\in ]-1, 1]\\
\frac 1t  & {\rm if}\ x\in ]1, 2t-1]\\
1-\frac{x-1}{2t} & {\rm if}\ x\in ] 2t-1, 2t+1]\ .
 \end{cases}
 \end{eqnarray*}
 Furthermore, $ZD[u_0](-t,-x)=ZD[u_0](t,x)$.
\end{theorem}
\begin{center}
\begin{figure}
\begin{tikzpicture}[scale=1.5]
\draw (5,0) node[below] {$x$};
\draw [->] (-5,0)--(5,0);
\draw (0,3) node[right] {$u_0(x)$};
\draw [->] (0,0)--(0,3);
\draw[red][thick](-1,1)--(1,1);
\draw[red][thick](-5,0)--(-1,0);
\draw[red][thick](1,0)--(5,0);
\draw[dotted](-1,0)--(-1,1);
\draw[dotted](1,0)--(1,1);
\draw (-1,0) node[below] {$-1$};
\draw (1,0) node[below] {$1$};
\draw (0,1.3) node[right] {$1$};
\end{tikzpicture}
\caption{The datum at $t=0$}
\end{figure}
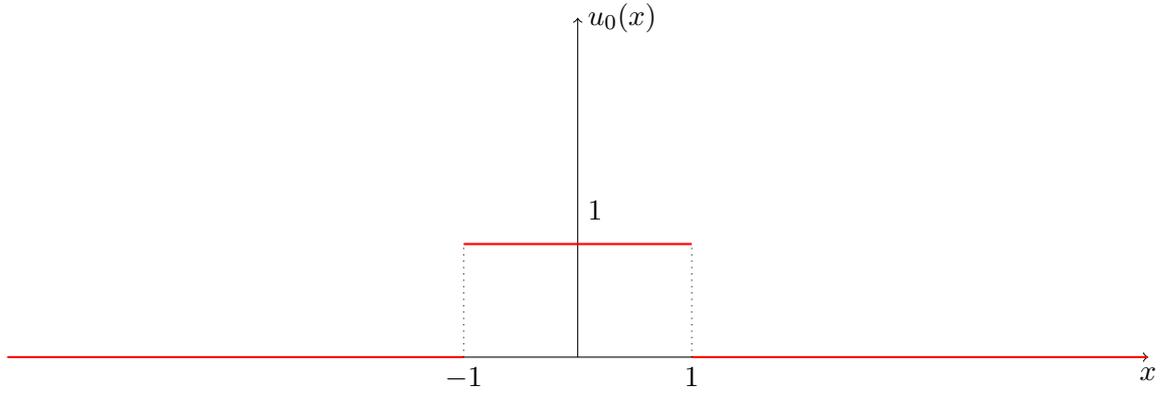
\end{center}
\begin{center}
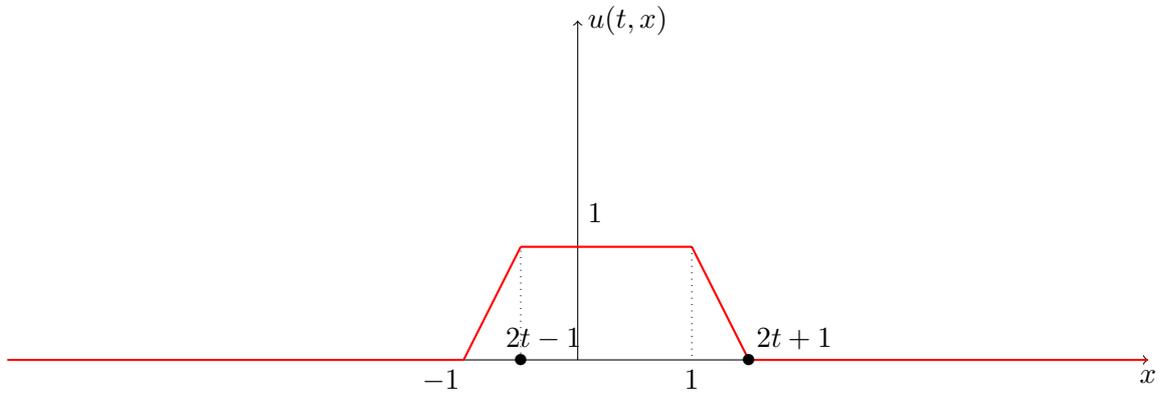
\begin{figure}
\begin{tikzpicture}[scale=1.5]
\draw (5,0) node[below] {$x$};
\draw [->] (-5,0)--(5,0);
\draw (0,3) node[right] {$u(t,x)$};
\draw [->] (0,0)--(0,3);
\draw[red][thick](-5,0)--(-1,0);
\draw[red][thick](-1,0)--(-0.5,1);
\draw[red][thick](-0.5,1)--(1,1);
\draw[red][thick](1,1)--(1.5,0);
\draw[red][thick](1.5,0)--(5,0);
\draw[dotted](-0.5,0)--(-0.5,1);
\draw[dotted](1,1)--(1,0);
\draw (-1.2,0) node[below] {$-1$};
\draw (-0.3,0) node[above] {$2t-1$};
\draw (-0.5,0) node {$\bullet$};
\draw (1.5,0) node {$\bullet$};
\draw (1,0) node[below] {$1$};
\draw (1.9,0) node[above] {$2t+1$};
\draw (0,1.3) node[right] {$1$};
\end{tikzpicture}
\caption{The weak limit  at $t\in ]0,1[$}
\end{figure}
\end{center}
\begin{center}
\begin{figure}
\begin{tikzpicture}[scale=1.5]
\draw (5,0) node[below] {$x$};
\draw [dotted] (0,1)--(1,1); 
\draw (0,1) node[left] {$1$};
\draw [->] (-5,0)--(5,0);
\draw (0,3) node[right] {$u(1,x)$};
\draw [->] (0,0)--(0,3);
\draw[red][thick](-5,0)--(-1,0);
\draw[red][thick](-1,0)--(1,1);
\draw[red][thick](1,1)--(3,0);
\draw[red][thick](3,0)--(5,0);
\draw[dotted](1,0)--(1,1);
\draw (-1,0) node[below] {$-1$};
\draw (1,0) node[below] {$1$};
\draw (3,0) node [below] {3};
\draw (3,0) node {$\bullet$};
\draw (1,0) node {$\bullet$};
\end{tikzpicture}
\caption{The weak limit   at $t=1$}
\end{figure}
\end{center}
\begin{center}
\begin{figure}
\begin{tikzpicture}[scale=1.5]
\draw (5,0) node[below] {$x$};
\draw [->] (-5,0)--(5,0);
\draw (0,3) node[right] {$u(t,x)$};
\draw [->] (0,0)--(0,3);
\draw[red][thick](-5,0)--(-1,0);
\draw[red][thick](-1,0)--(1,0.66);
\draw[red][thick](1,0.66)--(2,0.66);
\draw[red][thick](2,0.66)--(4,0);
\draw[red][thick](4,0)--(5,0);
\draw[dotted](2,0.66)--(2,0);
\draw[dotted](1,0.66)--(1,0);
\draw[dotted](0,0.66)--(1,0.66);
\draw (-1.2,0) node[below] {$-1$};
\draw (1,0) node[below] {$1$};
\draw (2,0) node[below] {$2t-1$};
\draw (4,0) node[below] {$2t+1$};
\draw (0,0.66) node[left] {$\frac 1t$};
\draw (2,0) node {$\bullet$};
\draw (4,0) node {$\bullet$};
\end{tikzpicture}
\caption{The weak limit    at $t\in ]1,+\infty [$}
\end{figure}
\end{center}
\begin{proof}
We use property (2) of Theorem \ref{main} and Theorem \ref{state:somalt}. We approximate $u_0$ by the following smooth function $u_0^\delta $ as $\delta \to 0^+$.
 \begin{eqnarray*}
 u_0^\delta (x)=\begin{cases}  0 &{\rm if} \ x\in ]-\infty ,-1-\delta ]\cup [1+\delta ,+\infty [\\
 1 &{\rm if} \ x\in [-1, 1] \end{cases}
 \end{eqnarray*}
 and $u_0^\delta $ is strictly increasing from $0$ to $1$ on the interval $[-1-\delta ,-1]$, strictly decreasing from $1$ to $0$ on the interval $[1,1+\delta ]$.
Let assume $t\in ]0,1]$. Given $x\in \R$, the solutions of the equation $y+2t u_0^\delta (y)=x$ are as follows.\\
If $x<-1-\delta$, the only solution is $y_0^\delta (t,x)=x$ and therefore $$ZD[u_0^\delta ](t,x)=u_0^\delta (y_0^\delta (t,x))=0\ .$$
If $-1 <x <2t-1$, there exists only one solution $y_0^\delta $, it is simple and it belongs to $]-1-\delta ,-1[$. Then
$$ZD[u_0^\delta ](t,x)=u_0^\delta (y_0^\delta (t,x))=\frac{x-y_0^\delta }{2t}\ .$$
If $2t-1<x<1$, there exists only one solution $y_0^\delta (t,x)=x-2t $, it is simple and it belongs to $]-1 ,1[$ . Then
$$ZD[u_0^\delta ](t,x)=u_0^\delta (y_0^\delta  (t,x))=1\ .$$
If $1+\delta <x<1+2t$ is outside  a set of measure $0$, there exist three simple solutions $y_0^\delta (t,x)=x-2t$, $y_1^\delta (t,x) \in ]1,1+\delta [$ and $y_2^\delta (t,x) =x$. Then
$$ZD[u_0^\delta ](t,x)=\sum_{k=0}^2 (-1)^k u_0^\delta (y_k^\delta (t,x))=1-\frac{x-y_1^\delta (t,x)}{2t}\ .$$
If $1+2t <x$, then there exists only one solution $y_0\delta (t,x)=x$ and 
$$ZD[u_0^\delta ](t,x)=u_0^\delta (y_0^\delta (t,x))=0\ .$$
Finally, passing to the limit as $\delta \to 0$ by keeping in mind that $$0\leq ZD[u_0^\delta ]\leq 1$$ in view of the maximum principle, we obtain the result.\\ The case $t>1$
 can be handled similarly. Finally, the formula $$ZD[u_0](-t,-x)=ZD[u_0](t,x)$$ follows from the fact that $u_0$ is even and therefore that the solution $u^\e $ of \eqref{BOeps} satisfies $u^\e (t,x)=u^\e (-t,-x)$.
\end{proof} 
We conclude this section by observing that the map $ZD $ does not satisfy the semigroup property. More precisely, we show on the above example that, in general,
$$ZD[ZD[u_0](t)](s)\ne ZD[u_0](t+s)\ .$$
Let us choose $t=1$ and $s>0$ in the above example. Then, according to Theorem \ref{step}, 
\begin{eqnarray*}
ZD[u_0](1,x)=\begin{cases} 0 & {\rm if}\ x\in ]-\infty ,-1]\cup ]3,+\infty[\\
\frac{x+1}{2} & {\rm if}\ x\in ]-1, 1]\\
1-\frac{x-1}{2} & {\rm if}\ x\in ] 1, 3]\ .
 \end{cases}
 \end{eqnarray*}
 Set $u_1(x):=ZD[u_0](1, x)$. This is a continuous, piecewise linear function, and we can easily approximate it by $C^1$ functions by smoothing it in neighbourhoods of $x=-1,1,3$. Given $s\in ]0,1[$, the function $y\mapsto y+2su_1(y)$ is continuous and strictly increasing, sending $]-\infty, -1[$ onto itself, $]-1, 1[$ onto $]-1,2s+1[$, $]1,3[$ onto 
 $]2s+1,3[$ and $]3,+\infty [$ onto itself. Consequently, we obtain, after passing to the limit,
 \begin{eqnarray*}
ZD[u_1](s,x)=\begin{cases} 0 & {\rm if}\ x\in ]-\infty ,-1[\cup ]3,+\infty[\\
\frac{x+1}{2(s+1)} & {\rm if}\ x\in ]-1, 2s+1[\\
\frac{3- x}{2(1-s)} & {\rm if}\ x\in ] 2s+1, 3]\\
0& {\rm if}\ x\in ] 3, +\infty[\ .
 \end{cases}
 \end{eqnarray*}
 \begin{center}
\begin{figure}
\begin{tikzpicture}[scale=1.5]
\draw (5,0) node[below] {$x$};
\draw [dotted] (0,1)--(2,1); 
\draw (0,1) node[left] {$1$};
\draw [->] (-5,0)--(5,0);
\draw (0,3) node[right] {$ZD[u_1](s,x)$};
\draw [->] (0,0)--(0,3);
\draw[red][thick](-5,0)--(-1,0);
\draw[red][thick](-1,0)--(2,1);
\draw[red][thick](2,1)--(3,0);
\draw[red][thick](3,0)--(5,0);
\draw[dotted](2,0)--(2,1);
\draw (-1,0) node[below] {$-1$};
\draw (2,0) node[below] {$2s+1$};
\draw (3,0) node [below] {3};
\draw (2,0) node {$\bullet$};
\draw (3,0) node {$\bullet$};
\end{tikzpicture}
\caption{The iteration  $ZD[ZD[u_0](1)](s)$ for $s\in ]0,1[$.}
\end{figure}
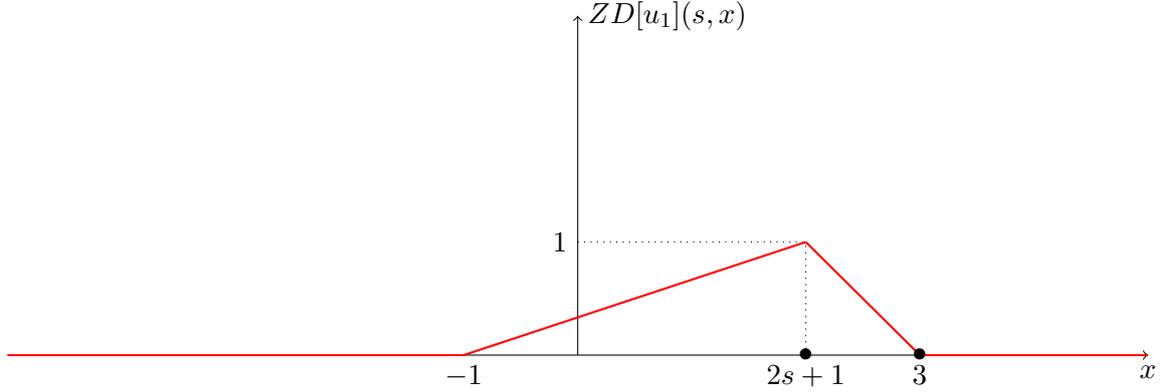
\end{center}
 Comparing with the expression of $ZD[u_0](1+s,x)$ from Theorem \ref{step}, we conclude that 
 $ZD[ZD[u_0](1)](s)\ne ZD[u_0](1+s)\ .$
 \begin{remark}
 As pointed to us by Y.~Brenier, this lack of semigroup property can also be proved by observing that the semigroup associated to the entropic solution \cite{K69, S99} of the inviscid Burgers--Hopf equation is deduced in \cite{YB81, YB83, YB84} from the map $ZD[\, .\,]$ through a Trotter formula. If the latter was a semigroup, it would coincide with the former, which is obviously wrong.
  \end{remark}

\end{document}